\documentclass[12pt]{amsart}
\usepackage{}
\usepackage{amsmath}
\usepackage{amsfonts}
\usepackage{amssymb}
\usepackage{bbding}
\usepackage{stmaryrd}
\usepackage{txfonts}
\usepackage{graphicx}
\usepackage{epsfig}
\usepackage{xypic}
\usepackage{tikz}
\usepackage{longtable}
\usepackage[all]{xy}
\usepackage{pgflibraryarrows}
\usepackage{pgflibrarysnakes}
\usepackage[shortlabels]{enumitem}
\usepackage{ifpdf}
\ifpdf
  \usepackage[colorlinks,final,backref=page,hyperindex]{hyperref}
\else
  \usepackage[colorlinks,final,backref=page,hyperindex,hypertex]{hyperref}
\fi

\newtheorem{theorem}{Theorem}[section]
\newtheorem{lemma}[theorem]{Lemma}

\newtheorem{prop}[theorem]{Proposition}
\theoremstyle{definition}
\newtheorem{defn}[theorem]{Definition}

\newtheorem{exam}[theorem]{Example}

\newcommand{\nc}{\newcommand}
\newcommand{\delete}[1]{}

\nc{\tred}[1]{\textcolor{red}{#1}} \nc{\tblue}[1]{\textcolor{blue}{#1}} \nc{\tgreen}[1]{\textcolor{green}{#1}} \nc{\tpurple}[1]{\textcolor{purple}{#1}} \nc{\btred}[1]{\textcolor{red}{\bf #1}} \nc{\btblue}[1]{\textcolor{blue}{\bf #1}} \nc{\btgreen}[1]{\textcolor{green}{\bf #1}} \nc{\btpurple}[1]{\textcolor{purple}{\bf #1}}

\renewcommand{\Bbb}{\mathbb}

\newcommand{\efootnote}[1]{}
\renewcommand{\textbf}[1]{}
\nc{\mlabel}[1]{\label{#1}}  % Use this to suppress names
\nc{\mcite}[2][]{\cite[#1]{#2}}  % Use this to suppress names
\nc{\mref}[1]{\ref{#1}}  % Use this to suppress names
\nc{\mbibitem}[1]{\bibitem{#1}} % Use this to show number
\delete{
\nc{\mlabel}[1]{\label{#1}  % Use the next two lines to show names
{\hfill \hspace{1cm}{\bf{{\ }\hfill(#1)}}}}
\nc{\mcite}[2][1]{\cite[#1]{#2}{{\bf{{\ }(#2; #1)}}}}  % Use this lines to show names
\nc{\mref}[1]{\ref{#1}{{\bf{{\ }(#1)}}}}  % Use this lines to show names
\nc{\mbibitem}[1]{\bibitem[\bf #1]{#1}} % Use this to show name
}

\renewcommand\geq{\geqslant}
\renewcommand\leq{\leqslant}

\renewcommand\bar[1]{\overline{#1}}
\nc{\into}{I}
\nc{\rbw}{\mathfrak{R}} \nc{\brp}{\mathrm{brp}} \nc{\lead}{\mathrm{Lead}} \nc{\Id}{\mathrm{Id}} \nc{\Irr}{\mathrm{Irr}}
\nc{\vx}{\sigma} \nc{\vy}{\tau} \nc{\dvx}{\sigma^{(1)}} \nc{\dvy}{\tau^{(1)}} \nc{\done}{\vep} \nc{\mcitep}[1]{\mcite{#1}} \nc{\wt}{\mathrm{wt}} \nc{\bre}[1]{|#1|} \nc{\mapmonoid}{\frakM} \nc{\disjoint}{\frakM'}
\nc{\ncpoly}[1]{\langle #1\rangle}  %for noncommutative polynomials
\nc{\mapm}[1]{\lfloor\!|{#1}|\!\rfloor}
\nc{\diff}[1]{{}^\NC\{ #1 \}} \nc{\disj}[1]{\{{#1}\}'} \nc{\mdisj}[1]{\frakM'(#1)} \nc{\brho}{\bar{\rho}} \nc{\om}{\bar{\frakm}} \nc{\frakn}{\mathfrak n} \nc{\ddeg}[1]{^{(#1)}} \nc{\opset}{X} \nc{\genset}{{Z}} \nc{\NC}{\mathrm{{NC}}} \nc{\leaf}{\mathrm{leaf}} \nc{\twig}{\mathrm{twig}} \nc{\fe}{\mathrm{fl}} \nc{\munderline}[1]{#1} \nc{\bo}{o} \nc{\dep}{\mathrm{depth}} \nc{\ofe}{\mathrm{ofl}} \nc{\dfe}{\mathrm{dfe}} \nc{\fex}{\mathrm{fex}} \nc{\dl}{\mathrm{dlex}} \nc{\db}{\mathrm{db}} \nc{\lex}{\mathrm{lex}} \nc{\clex}{\mathrm{clex}} \nc{\dgp}{\mathrm{dgp}} \nc{\dgx}{\mathrm{dgx}} \nc{\br}{\mathrm{br}} \nc{\obd}{\mathrm{odb}} \nc{\ob}{\mathrm{ob}}
\nc{\pie}{\mathrm{PIE}}
\nc{\rbo}{\mathrm{RBO}}
\nc{\supp}{\mathcal{S}}
\nc{\nul}{\mathcal{Z}}
\nc{\bin}[2]{ (_{\stackrel{\scs{#1}}{\scs{#2}}})}  %binomial coeff
\nc{\binc}[2]{ \left (\!\! \begin{array}{c} \scs{#1}\\
    \scs{#2} \end{array}\!\! \right )}  %binomial coeff
\nc{\bincc}[2]{  \left ( {\scs{#1} \atop
    \vspace{-1cm}\scs{#2}} \right )}  %binomial coeff
\nc{\bs}{\bar{S}} \nc{\cosum}{\sqsubset} \nc{\la}{\longrightarrow} \nc{\rar}{\rightarrow} \nc{\dar}{\downarrow} \nc{\dprod}{**} \nc{\dap}[1]{\downarrow \rlap{$\scriptstyle{#1}$}} \nc{\md}[1]{\bar{#1}} \nc{\uap}[1]{\uparrow \rlap{$\scriptstyle{#1}$}} \nc{\defeq}{\stackrel{\rm def}{=}} \nc{\disp}[1]{\displaystyle{#1}} \nc{\dotcup}{\ \displaystyle{\bigcup^\bullet}\ } \nc{\gzeta}{\bar{\zeta}} \nc{\hcm}{\ \hat{,}\ } \nc{\hts}{\hat{\otimes}} \nc{\barot}{{\otimes}} \nc{\free}[1]{\bar{#1}} \nc{\uni}[1]{\tilde{#1}} \nc{\hcirc}{\hat{\circ}} \nc{\leng}{\ell} \nc{\lleft}{[} \nc{\lright}{]} \nc{\lc}{\lfloor} \nc{\rc}{\rfloor}
\nc{\lb}{[} %left bracket
\nc{\rb}{]} %right bracket
\nc{\curlyl}{\left \{ \begin{array}{c} {} \\ {} \end{array}
    \right.  \!\!\!\!\!\!\!}
\nc{\curlyr}{ \!\!\!\!\!\!\!
    \left. \begin{array}{c} {} \\ {} \end{array}
    \right \} }
\nc{\longmid}{\left | \begin{array}{c} {} \\ {} \end{array}
    \right. \!\!\!\!\!\!\!}
\nc{\onetree}{\bullet} \nc{\ora}[1]{\stackrel{#1}{\rar}}
\nc{\ola}[1]{\stackrel{#1}{\la}}%${\Bbb Z}$
\nc{\ot}{\otimes} \nc{\mot}{{{\boxtimes\,}}} \nc{\otm}{\overline{\boxtimes}} \nc{\sprod}{\bullet} \nc{\scs}[1]{\scriptstyle{#1}} \nc{\mrm}[1]{{\rm #1}} \nc{\msum}{\sum\limits}
\nc{\margin}[1]{\marginpar{\rm #1}}   %{\rm #1}}
\nc{\dirlim}{\displaystyle{\lim_{\longrightarrow}}\,} \nc{\invlim}{\displaystyle{\lim_{\longleftarrow}}\,} \nc{\mvp}{\vspace{0.3cm}} \nc{\tk}{^{(k)}} \nc{\tp}{^\prime} \nc{\ttp}{^{\prime\prime}} \nc{\svp}{\vspace{2cm}} \nc{\vp}{\vspace{8cm}} \nc{\proofbegin}{\noindent{\bf Proof: }}
\nc{\proofend}{$\blacksquare$ \vspace{0.3cm}}
\nc{\modg}[1]{\!<\!\!{#1}\!\!>}
\nc{\intg}[1]{F_C(#1)} \nc{\lmodg}{\!<\!\!} \nc{\rmodg}{\!\!>\!} \nc{\cpi}{\widehat{\Pi}}
\nc{\sha}{{\mbox{\cyr X}}}  %used to be \cyr
\nc{\shap}{{\mbox{\cyrs X}}} %sha as product
\nc{\shpr}{\diamond}    %Shuffle product
\nc{\shp}{\ast} \nc{\shplus}{\shpr^+}
\nc{\shprc}{\shpr_c}    %Cartier's product
\nc{\msh}{\ast} \nc{\zprod}{m_0} \nc{\oprod}{m_1} \nc{\vep}{\varepsilon} \nc{\labs}{\mid\!} \nc{\rabs}{\!\mid}
\nc{\astarrow}{\overset{\raisebox{-3pt}{$\ast$}}{\rightarrow}}
\nc{\dth}{d} \nc{\mmbox}[1]{\mbox{\ #1\ }} \nc{\fp}{\mrm{FP}} \nc{\rchar}{\mrm{char}} \nc{\Fil}{\mrm{Fil}} \nc{\Mor}{Mor\xspace} \nc{\gmzvs}{gMZV\xspace} \nc{\gmzv}{gMZV\xspace} \nc{\mzv}{MZV\xspace} \nc{\mzvs}{MZVs\xspace} \nc{\Hom}{\mrm{Hom}} \nc{\id}{\mrm{id}} \nc{\im}{\mrm{im}} \nc{\incl}{\mrm{incl}} \nc{\map}{\mrm{Map}} \nc{\mchar}{\rm char} \nc{\nz}{\rm NZ}

\nc{\Alg}{\mathbf{Alg}} \nc{\Bax}{\mathbf{Bax}} \nc{\bff}{\mathbf f} \nc{\bfk}{{\bf k}} \nc{\bfone}{{\bf 1}} \nc{\bfx}{\mathbf x} \nc{\bfy}{\mathbf y}
\nc{\base}[1]{\bfone^{\otimes ({#1}+1)}} %{{a_{#1}}}
\nc{\Cat}{\mathbf{Cat}} \delete{}
\nc{\detail}{\marginpar{\bf More detail}
    \noindent{\bf Need more detail!}
    \svp}
\nc{\Int}{\mathbf{Int}} \nc{\Mon}{\mathbf{Mon}}
\nc{\rbtm}{{shuffle }} \nc{\rbto}{{Rota-Baxter }} \nc{\remarks}{\noindent{\bf Remarks: }} \nc{\Rings}{\mathbf{Rings}} \nc{\Sets}{\mathbf{Sets}}

\nc{\BA}{{\Bbb A}} \nc{\CC}{{\Bbb C}} \nc{\DD}{{\Bbb D}} \nc{\EE}{{\Bbb E}} \nc{\FF}{{\Bbb F}} \nc{\GG}{{\Bbb G}} \nc{\HH}{{\Bbb H}} \nc{\LL}{{\Bbb L}} \nc{\NN}{{\Bbb N}} \nc{\KK}{{\Bbb K}} \nc{\QQ}{{\Bbb Q}} \nc{\RR}{{\Bbb R}} \nc{\TT}{{\Bbb T}} \nc{\VV}{{\Bbb V}} \nc{\ZZ}{{\Bbb Z}}

\nc{\cala}{{\mathcal A}} \nc{\calc}{{\mathcal C}} \nc{\cald}{{\mathcal D}} \nc{\cale}{{\mathcal E}} \nc{\calf}{{\mathcal F}} \nc{\calg}{{\mathcal G}} \nc{\calh}{{\mathcal H}} \nc{\cali}{{\mathcal I}} \nc{\call}{{\mathcal L}} \nc{\calm}{{\mathcal M}} \nc{\caln}{{\mathcal N}} \nc{\calo}{{\mathcal O}} \nc{\calp}{{\mathcal P}} \nc{\calr}{{\mathcal R}} \nc{\cals}{{\mathcal S}} \nc{\calt}{{\mathcal T}} \nc{\calw}{{\mathcal W}} \nc{\calk}{{\mathcal K}} \nc{\calx}{{\mathcal X}}
\nc{\calz}{{\mathcal Z}}
 \nc{\CA}{\mathcal{A}}

\nc{\fraka}{{\mathfrak a}} \nc{\frakA}{{\mathfrak A}} \nc{\frakb}{{\mathfrak b}} \nc{\frakB}{{\mathfrak B}}
\nc{\frakc}{{\mathfrak c}}  \nc{\frakD}{{\mathfrak D}}
\nc{\frakH}{{\mathfrak H}}
\nc{\frakh}{{\mathfrak h}} \nc{\frakM}{{\mathfrak M}}
\nc{\frakO}{{\mathfrak O}}
\nc{\frakE}{{\mathfrak E}}
\nc{\bfrakM}{\overline{\frakM}} \nc{\frakm}{{\mathfrak m}} \nc{\frakP}{{\mathfrak P}} \nc{\frakN}{{\mathfrak N}} \nc{\frakp}{{\mathfrak p}} \nc{\frakS}{{\mathfrak S}}
\nc{\frakk}{{\mathfrak k}}
\nc{\frakx}{{\mathfrak x}}
\nc{\frakl}{{\mathfrak l}} \nc{\ox}{\bar{\frakx}} \nc{\frakX}{{\mathfrak X}} \nc{\fraky}{{\mathfrak y}} \nc\dop{\delta}
\nc{\Reduce}{{\rm Red}}

\font\cyr=wncyr10 \font\cyrs=wncyr7
\nc{\redt}[1]{\textcolor{red}{#1}}
\nc{\ma}[1]{\textcolor{green}{\tt Markus:#1}}
\nc{\li}[1]{\textcolor{red}{\tt Li:#1}} \nc{\sz}[1]{\textcolor{blue}{\tt sz:#1}} \nc{\xg}[1]{\textcolor{purple}{\tt xg:#1}}
\nc{\nonz}[1]{#1^\times}
\nc{\NNP}{\nonz{\NN}}
\nc{\stdint}[1]{J_{#1}}
\nc{\dualmod}[1]{#1^*}
\nc{\alghom}[1]{#1^\bullet}
\nc{\End}{\mathrm{End}}
\nc{\rng}{\mathrm{\mathcal{R}}}
\nc{\codim}{\mathrm{codim}}
\nc{\evl}{\mathrm{ev}}

\nc{\oh}{\,\small{\textcircled{\tiny{$\frakH$}}}\,}
\nc{\ohs}{\small{\textcircled{\tiny{$\frakH$}}}}
\nc{\ohzs}{\,\small{\textcircled{\tiny{$\frakh$}}}\,}
\nc{\ohz}{\small{\textcircled{\tiny{$\frakh$}}}}
\nc{\ug}{\mathfrak{U}_h(\mathfrak{g})}
\nc{\frakg}{\mathfrak{g}}
\nc{\frakU}{\mathfrak U}
\nc{\brag}{[,]_\frakg}

\begin{document}
\title{Free involutive Hom-semigroups and Hom-associative algebras}

\author{Li Guo}
\address{
%Department of Mathematics, Lanzhou University, Lanzhou, Gansu 730000, China and
Department of Mathematics and Computer Science, Rutgers University, Newark, NJ 07102, USA}
\email{liguo@rutgers.edu}

\author{Shanghua Zheng}
\address{Department of Mathematics, Lanzhou University, Lanzhou, Gansu 730000, China}
\email{zheng2712801@163.com}

%========================================================================
\hyphenpenalty=8000
\date{\today}

\begin{abstract}
In this paper we construct free Hom-semigroups when its unary operation is multiplicative and is an involution. Our method of construction is by bracketed words. As a consequence, we obtain free Hom-associative algebras generated by a set under the same conditions for the unary operation.
\end{abstract}
%\subjclass[2010]{16W99, 45N05, 47G10, 12H20}

%\keywords{Rota-Baxter operator, averaging operator, integration, monomial linear operator}

\maketitle

\tableofcontents

\hyphenpenalty=8000 \setcounter{section}{0}

%========================================================================

\section{Introduction}\mlabel{sec:int}

A Hom-Lie algebra is a generalization of a Lie algebra. It consists of a vector space $L$, a bilinear skew-symmetric bracket $[\cdot, \cdot ]: L\ot L\to L$ and a linear self-map $\alpha:L\to L$ such that the {\bf Hom-Jacobi identity} holds
$$ [\alpha(x),[y,z]]+[\alpha(y),[z,x]]
+[\alpha(z),[x,y]]=0.$$
Thus when $\alpha$ is the identity map, a Hom-Lie algebra is just a Lie algebra. The concept of a Hom-Lie algebra was introduced in~\cite{HLS} to describe the structures on certain deformations of the Witt algebra and the Virasoro algebra. But related constructions could already be found in earlier literature~\cite{AS,Hu,Liu}.

As is well-known, an associative algebra gives a Lie algebra by taking the commutator bracket. To give a similar approach to Hom-Lie algebra, Makhlouf and Silvestrov~\cite{MS2} introduced the concept of a Hom-associative algebra $(A,\mu, \alpha)$ in which the binary operation $\mu$
satisfies an $\alpha$-twisted version of the associativity. Afterwards, Yau~\cite{Yau4} constructed the enveloping Hom-associative algebra of a Hom-Lie algebra. Since then the concepts of various other Hom-structures have been introduced with broad connections in mathematics and mathematical physics~\cite{HMS,MS2}. See~\cite{Mak,MY,MS1,SheBai,Yau1,Yau2,Yau3,Yau4} for further results in this direction.

While it did not pose much difficulty in establishing the concepts of Hom-generalizations for many algebraic structures, their studies often turned out be much more challenging than their classical counter parts. One case in point is the construction of free objects. Even though explicit constructions of many algebraic structures are known, such constructions have been established for very few Hom-algebraic structures. Even the explicit construction of the most classical algebraic structure, namely of free Hom-associative algebras are not known. See~\cite{Yau4} for the construction of free Hom-associative algebras as quotients of free Hom-nonassociative algebras.

Note that the construction of a free algebra on a set essentially comes from the construction of a free semigroup on the set. From this viewpoint, we introduce the concept of a Hom-semigroup in this paper and provide an explicit construction of free objects for a special yet important class of Hom-semigroups, namely the involutive Hom-semigroups. In fact in the literature, the involutive condition is often assumed for a Hom-Lie algebra or Hom-associative algebra~\cite{SheBai}. As a consequence we obtain an explicit construction of free involutive Hom-associative algebras. After giving the concepts and basic examples or Hom-semigroups, we first construct free involutive Hom-semigroups in Theorem~\mref{thm:freeb}. We then obtain the construction of a free involutive Hom-associative algebra on a set by a simple linear span of the free involutive Hom-semigroup. This is given in Theorem~\mref{thm:freex}. Further study of Hom-associative algebras and Hom-Lie algebras will be continued in a future work.

\section{Involutive Hom-semigroups}
\mlabel{sec:inv}

In this section, we introduce the notions of Hom-semigroup and involutive Hom-semigroup, and give some examples.

%The definition of a Hom-semigroup generalizes the classical semigroup by twisting the associative law by a set map.

\begin{defn}
\begin{enumerate}
\item
A {\bf Hom-semigroup} is a set $S$ together with a binary operation $\mu$ (that will often be suppressed from the notation: $\mu(x,y)=xy$) and a unary operation $\alpha$ on $S$ that satisfy the {\bf Hom-associative law}:
$$\alpha(x)(yz)=(xy)\alpha(z),\quad\text{
for all $x,y,z\in S$}.$$
\item
A Hom-semigroup $(S,\mu,\alpha)$ is called {\bf multiplicative} if $\alpha(xy)=\alpha(x)\alpha(y)$
for all $x,y\in S$.
\item
A Hom-semigroup $(S,\mu,\alpha)$ is called {\bf involutive} if it is multiplicative and $\alpha^2=\id$, the identity map on $S$.
\item
Let $S:=(S,\mu,\alpha)$ and $S':=(S',\mu',\alpha')$ be two Hom-semigroups. A set map $f:S\to S'$ is called a {\bf morphism of Hom-semigroups} if $f(\mu(x,y))=\mu'(f(x),f(y))$ and $f(\alpha(x))=\alpha'(f(x))$ for all $x,y\in S$. A morphism of Hom-semigroups $f$ is called an {\bf isomorphism} if $f$ is a bijection.
\end{enumerate}
\end{defn}

By taking $\alpha=\id$ in a Hom-semigroup, we see that every semigroup is not only a Hom-semigroup but also an involutive Hom-semigroup. But a Hom-semigroup  is not necessarily a semigroup in general, as we can see from the following example.
\begin{exam}Let $S=\{x,y,z\,\}$. Define a binary product on $S$ by the following Cayley table:
$$\begin{tabular}{c|ccccc}
$\cdot$ &$x$ & $y$& $z$\\
\hline
$x$ &$y$ & $x$ &$z$\\
$y$ &$y$ & $y$ &$z$\\
$z$ &$z$ & $z$ &$z$\\
\end{tabular}
$$
Further define a set map $\alpha: S\to S$ by
$$\alpha(x)=\alpha(y)=\alpha(z)=z.$$
Then we can check that $(S, \cdot, \alpha)$ is a Hom-semigroup. But since $(xy)x=xx=y$ and $x(yx)=xy=x$, $(S,\cdot)$ is not a semigroup.
\end{exam}

We next  construct a Hom-semigroup from any given semigroup $S$. If a semigroup $S$ with at least two elements contains an element $0$ such that
$$x0=0x=0,\quad\text{for all $x\in S$},$$
then the element $0$ is called a {\bf zero element} of $S$, and that $S$ is a \emph{semigroup with zero}. As can be easily checked, a zero element of a semigroup is unique. If $S$ has no zero element, we can adjoin a new element $0$ to $S$ and define
$$x0=0x=00=0,\quad\text{for all $x\in S$}.$$
Then the associativity still holds in the extended set $S\cup \{0\}$, making it a semigroup with zero.
We define
$$S^0:=\left\{\begin{array}{lll}
S, &\text{if $S$ has a zero element},\\
S\cup \{0\}, &\text{otherwise}.
\end{array}\right.$$
With this notation, we have
\begin{prop}
Let $S$ be a semigroup. Define a set map $\alpha_0: S^0\to S^0$ by $\alpha_0(x):=0$ for all $x\in S^0$. Then $S^0$ with the set map $\alpha_0$ is a Hom-semigroup.
\end{prop}
\begin{proof}
We check that $\alpha(x)(yz)=0=(xy)\alpha(z)$ holds for all $x,y,z\in S$. Thus $(S,\alpha)$ is a Hom-semigroup.
\end{proof}

We next give an example of an involutive Hom-semigroup.
\begin{exam}
Let $S=\{x,y,z\,\}$. Define a binary product on $S$ by the following Cayley table.
$$\begin{tabular}{c|ccccc}
$\cdot$ &$x$ & $y$& $z$\\
\hline
$x$ &$y$ & $x$ &$z$\\
$y$ &$y$ & $x$ &$z$\\
$z$ &$z$ & $z$ &$z$\\
\end{tabular}
$$
We further define a set map $\alpha: S\to S$ by taking $$\alpha(x)=y,\,\alpha(y)=x,\,\alpha(z)=z.$$
Then we have $\alpha^2=\id$ and $\alpha(ab)=\alpha(a)\alpha(b)$ for all $a,b\in S$. In order to prove that $(S,\cdot,\alpha)$ is an involutive Hom-semigroup, we only need to verify the Hom-associative law:
$$\alpha(a)(bc)=(ab)\alpha(c)\quad\text
{for all $a,b,c\in S$.}$$
By the above Cayley table and $\alpha(z)=z$, the Hom-associativity holds if one of $a,b,c$ is taken to be $z$. Thus it remains to prove that the Hom-associativity holds for $a,b,c\in \{x,y\}.$ We divide into the following eight  cases to consider.
\begin{eqnarray*}
a=b=c=x;\quad  a=b=x, c=y;\quad
a=c=x, b=y;\quad b=c=x,a=y;\\
a=b=c=y;\quad
a=b=y, c=x;\quad
a=c=y, b=x;\quad
b=c=y,a=x.
\end{eqnarray*}
The verification of the Hom-associativity in each case is simple. For example, taking $a=c=x,b=y$ we have $\alpha(x)(yx)=yy=x$ and $(xy)\alpha(x)=xy=x$.  Then $\alpha(x)(yx)=(xy)\alpha(x).$
\end{exam}

\delete{
Let $S$ be a non-empty set. Let $M_n(S)$ denote the set of  $n\times n$ matrices with entries in $S$.
\begin{exam}\cite{Yau3}
Let $(S,\mu,\alpha)$ be a  Hom-semigroup.  Then $(M_n(S), \mu', \alpha')$ is a  Hom-semigroup, where the multiplication $\mu'$ is given by matrix multiplication and $\alpha'$ is given by $\alpha$ in each entry. Moreover, if $(S,\mu,\alpha)$ is an involutive Hom-semigroup, then $(M_n(S),\mu',\alpha')$ is also an involutive Hom-semigroup.
\end{exam}
}

\section{Free involutive Hom-semigroups}
\mlabel{sec:free}
In this section, we construct the free involutive Hom-semigroup generated by a set.
The construction will be given by bracketed words. We will carry out the construction in Section~\mref{ss:cons} that leads to Theorem~\mref{thm:freeb}, our main result of this paper. We then provide the proof of this theorem in Section~\mref{ssec:proof}.
\subsection{The construction by bracketed words}
\mlabel{ss:cons}

We start with the basis definition of the free involutive Hom-semigroup on a set.
\begin{defn}
A {\bf free involutive Hom-semigroup on a set $X$} is an involutive Hom-semigroup $(F(X),\ast,\alpha)$ with a set map $j_X:X\to F(X)$ such that, for any involutive Hom-semigroup $(S,\cdot,\beta)$ and any set map $f:X\to S$, there is a unique homomorphism $\bar{f}:F(X)\to S$ of  Hom-semigroups such that $\bar{f}\circ  j_X=f$.
\end{defn}

We first start with a construction of the set of the free involutive Hom-semigroup on $X$ by bracketed words\cite{Gop}.
Let $\lc X\rc$ denote the set $\{\lc x\rc\,|\, x\in X\}$.
Thus $\lc X\rc$ is a set that is indexed by $X$ but disjoint with $X$. Also denote $\lc X\rc^{(0)}=X$, $\lc x\rc^{(0)}=x$, $\lc X\rc^{(1)}=\lc X\rc$ and $\lc x\rc^{(1)}=\lc x\rc$ by convention.
We then define the set $$\tilde{X}:=\lc X\rc ^{(0)} \sqcup \lc X\rc= X\sqcup \lc X\rc=\{\lc x\rc^{(k)}\,|\,x\in X, k\in \{0,1\}\}.$$

Define
$$\calh(X):=\bigcup_{n\geq 1}\tilde{X}^n.$$
Thus $\calh(X)$ has the same underlying set as the free semigroup generated by $\tilde{X}$. So any $\frakx\in \tilde{X}^n$ is of the form
$$\frakx=\lc x_1\rc^{(k_1)}\cdots\lc x_{n-1}\rc^{(k_{n-1})}\lc x_{n}\rc^{(k_n)},$$
where $x_i\in X$ and $k_i\in \{0,1\}$ for $i=1,\cdots,n$. But instead of the usual concatenation multiplication for the free semigroup, we will define a different multiplication on $\calh(X)$.

Let $m\in \NN$. Then we denote $\bar{m}:=m\pmod 2$ for the residue of $m$ modulo 2. Define a set map $\alpha_X$ on $\calh(X)$ by defining
$$\alpha_X(\frakx)=\lc x_1\rc^{(\bar{k_1+1})}\cdots\lc x_{n-1}\rc^{(\bar{k_{n-1}+1})}\lc x_n\rc^{(\bar{k_n+1})} \quad
\text{for all } \frakx=\lc x_1\rc^{(k_1)}\cdots\lc x_{n-1}\rc^{(k_{n-1})}\lc x_{n}\rc^{(k_n)} \in \tilde{X}^n.$$
Since $\bar{k_i+1}$ is also in $\{0,1\}$ for $1\leq i\leq n$,  $\alpha_X(\frakx)$ is in $\tilde{X}^n$. Thus the set map $\alpha_X$ is well-defined. For $n=1$, we have

\begin{equation}\alpha_X(\lc x\rc^{(k)})=\lc x\rc^{(\bar{k+1})} \quad\text{for all}\,\, \lc x\rc^{(k)}\in \tilde{X}.
\mlabel{eq:alphx}
\end{equation}
So in particular, $\alpha_X(x)=\lc x \rc$ for all $x\in X$. Then
we obtain
\begin{equation}
\alpha_X(\frakx)=\alpha_X(\lc x_1\rc^{(k_1)})\cdots\alpha_X(\lc x_{n-1}\rc^{(k_{n-1})})\alpha_X(\lc x_{n}\rc^{(k_n)}).
\mlabel{eq:br1}
\end{equation}
By the definition of $\alpha_X$,  we have
\begin{equation}
\alpha_X^2=\id.
\mlabel{eq:bramap}
\end{equation}

We  next define a product $\diamond$ on $\calh(X)$ by defining
$\frakx\diamond\frakx'$ for all $\frakx\in \tilde{X}^i, \frakx'\in \tilde{X}^j$ where $i, j\geq 1$. We achieve this by applying induction on the sum $n:=i+j\geq 2$. If $n=2$, then $\frakx,\frakx'$ are in $\tilde{X}$, and then we define $\frakx\diamond\frakx':=
\frakx\frakx'$, the concatenation of $\frakx$ and $\frakx'$.
Suppose $\frakx \diamond\frakx'$ have been defined for $\frakx$ and $\frakx'$ with $n\leq p$,
and consider $\frakx\in\tilde{X}^i$ and $\frakx'\in\tilde{X}^j$ with $n=p+1$.
Write $\frakx=\lc x_1\rc^{(k_1)}\cdots\lc x_{i}\rc^{(k_i)}$ and  $\frakx'=\lc x'_1\rc^{(t_1)}\cdots\lc x'_{j}\rc^{(t_j)}$.
Then we define
\begin{equation}
\frakx\diamond\frakx':=
\left\{\begin{array}{llll}
\lc x_1\rc^{(k_1)}\lc x'_1\rc^{(t_1)}\cdots\lc x'_{j}\rc^{(t_j)}, &\text{if}\,\, i=1,\\
\lc x_1\rc^{(\bar{k_1+1})}\left(\lc x_2\rc^{(k_2)} \cdots\lc x_i\rc^{(k_i)}\diamond \alpha_X(\lc x'_1\rc^{(t_1)}\cdots
\lc x'_{j}\rc^{(t_j)})\right),&\text{if}\,\, i\geq 2 .
\end{array}\right.
\mlabel{eq:prod}
\end{equation}
Here the product of $\frakx$ and $\frakx'$ in the first case is by concatenation. Since $\lc x_2\rc^{(k_2)}\cdots\lc x_i\rc^{(k_i)}$ is in $\tilde{X}^{i-1}$ and $\alpha_X(\lc x'_1\rc^{(t_1)}\cdots
\lc x'_{j}\rc^{(t_j)})$ is in $\tilde{X}^j$, the sum $i+j-1$ is $p$. Then by the induction hypothesis, the second case is well-defined. As a consequence, we obtain the following alternative description of the product.
\begin{equation}
\frakx\diamond\frakx'=\lc x_1\rc^{(\bar{k_1+1})}\lc x_2\rc^{(\bar{k_2+1})} \cdots\lc x_{i-1}\rc^{(\bar{k_{i-1}+1})}\lc x_i\rc^{(k_i)}\alpha_X^{i-1}(
\frakx').
\mlabel{eq:prodex}
\end{equation}
\begin{lemma}
Let $\frakx\in \tilde{X}^i$ with $i>1$ and let $\frakx=\frakx_1\frakx_2$ where $\frakx_p\in \tilde{X}^{i_p}$ with $i_p>0$ for $p=1,2$. Then for $\frakx'\in \tilde{X}^j$ we have
\begin{equation}
\frakx\diamond \frakx' = \alpha_X(\frakx_1)(\frakx_2\diamond \alpha_X^{i_1}(\frakx')).
\mlabel{eq:prodrec}
\end{equation}
\mlabel{lem:prodrec}
\end{lemma}
\begin{proof} Let $\frakx=\frakx_1\frakx_2\in\tilde{X}^i$ as in the lemma. Then we have $i_1+i_2=i$.
We use induction on $i\geq 2$.  For $i=2$, we have $i_1=i_2=1$. Then $\frakx_1,\frakx_2$ are in $\tilde{X}$. By Eqs.~(\mref{eq:alphx}) and~(\mref{eq:prod}), we obtain
$$\frakx\diamond\frakx'=(\frakx_1\frakx_2)
\diamond\frakx'=\alpha_X(\frakx_1)(\frakx_2
\diamond \alpha_X(\frakx')).$$
Then Eq.~(\mref{eq:prodrec}) holds. Suppose Eq.~(\mref{eq:prodrec}) has been proved for $i\geq 2$ and consider  $\frakx\in\tilde{X}^{i+1}$. Then there are two cases.
\smallskip

\noindent{\bf Case 1. $i_1=1$:} Then by Eqs.~(\mref{eq:alphx}) and~(\mref{eq:prod}) again, we have
$$\frakx\diamond\frakx'=(\frakx_1\frakx_2)
\diamond\frakx'=\alpha_X(\frakx_1)(\frakx_2
\diamond\alpha_X(\frakx')).$$
Thus Eq.~(\mref{eq:prodrec}) holds.
\smallskip

\noindent{\bf Case 2. $i_1>1$:} Then we write $\frakx_1=\lc x\rc^{(k)}\hat{\frakx}_1$, where $\lc x\rc^{(k)}\in \tilde{X}$ and $\hat{\frakx}_1\in \tilde{X}^{i_1-1}$.
Then we have
\begin{eqnarray*}
\frakx\diamond\frakx'&=&
(\lc x\rc^{(k)}\hat{\frakx}_1\frakx_2)
\diamond\frakx'\\
&=&\lc x\rc^{(\bar{k+1})}(\hat{\frakx}_1\frakx_2
\diamond\alpha_X(\frakx'))\quad\text{(by Eq.~(\mref{eq:prod})})\\
&=&\alpha_X(\lc x\rc^{(k)})\Big(\hat{\frakx}_1\frakx_2
\diamond\alpha_X(\frakx')\Big)
\quad\text{(by Eq.~(\mref{eq:alphx})})\\
&=&\alpha_X(\lc x\rc^{(k)})\left(\alpha_X(\hat{\frakx}_1)
\big(\frakx_2\diamond\alpha_X^{i_1-1}
(\alpha_X(\frakx'))\big)\right)\quad\text{(by the induction hypothesis)}\\
&=&\alpha_X(\lc x\rc ^{(k)})\alpha_X(
\hat{\frakx}_1)\left(\frakx_2\diamond
\alpha_X^{i_1}(\frakx')\right)\\
&=&\alpha_X(\lc x\rc^{(k)}\hat{\frakx}_1)\left(\frakx_2\diamond
\alpha_X^{i_1}(\frakx')\right)
\quad\text{(by Eq.~(\mref{eq:br1}))}\\
&=&\alpha_X(\frakx_1)(\frakx_2\diamond
\alpha_X^{i_1}(\frakx')).
\end{eqnarray*}
This completes the inductive proof of Eq.~(\mref{eq:prodrec}).
\end{proof}
By Eq.~(\mref{eq:prod}), we have
\begin{equation}
\frakx \diamond \frakx'=\frakx\frakx'\quad
\text{for all } \frakx\in\tilde{X},\,
\frakx'\in\tilde{X}^n.
\mlabel{eq:diam}
\end{equation}
Further by Eq.~(\mref{eq:br1}),  we get
$$\alpha_X(\frakx\diamond \frakx')=\alpha_X(\frakx\frakx')=\alpha_X(\frakx )\alpha_X(\frakx') \quad \text{ for all } \frakx\in\tilde{X}, \frakx'\in\tilde{X}^n.$$
Thus
\begin{equation}
\alpha_X(\frakx \diamond \frakx')=\alpha_X(\frakx) \diamond \alpha_X( \frakx') \quad\text{ for all } \frakx\in \tilde{X}, \frakx'\in \tilde{X}^n.
\mlabel{eq:combra}
\end{equation}

We now state our main result on the free involutive Hom-semigroup constructed by bracketed words. It will be proved in the next subsection.
\begin{theorem} Let $X$ be a set. Let $j_X: X\to \calh(X)$ be the inclusion map.
\begin{enumerate}
\item
The triple $(\calh(X),\diamond,\alpha_X)$ is an involutive Hom-semigroup.
\mlabel{it:hombra}
\item
The quadruple $(\calh(X),\diamond,\alpha_X, j_X)$ is the free involutive Hom-semigroup on $X$.
\mlabel{it:freebra}
\end{enumerate}
\mlabel{thm:freeb}
\end{theorem}

\subsection{The proof of Theorem~\mref{thm:freeb}}
\mlabel{ssec:proof}
We now prove Theorem~\mref{thm:freeb}.
\subsubsection{The proof of Theorem~\mref{thm:freeb}
(\mref{it:hombra})}
By Eq.~(\mref{eq:bramap}), we have $\alpha_X^2=\id$. We next prove
\begin{equation}
\alpha_X( \frakx \diamond\frakx')
=\alpha_X(\frakx)\diamond\alpha_X(\frakx')
\quad\text {for all } \frakx\in\tilde{X}^i,
\frakx'\in\tilde{X}^j.
\mlabel{eq:comdia}
\end{equation}
For this we apply induction on the sum $i+j\geq 2$. When $i+j=2$, we have $\frakx,\frakx'\in\tilde{X}$. By Eq.~(\mref{eq:combra}), we have
$\alpha_X(\frakx\diamond\frakx')=\alpha_X( \frakx)\diamond\alpha_X( \frakx')$.
Assume that Eq.~(\mref{eq:comdia}) holds for $i+j\leq p$ with $p\geq 2$. Consider $\frakx\in \tilde{X}^i$ and $\frakx'\in \tilde{X}^j$ with $i+j=p+1$. If $i=1$, then $\frakx\in \tilde{X}$. By Eq.~(\mref{eq:combra}) again, Eq.~(\mref{eq:comdia}) holds.
If $i\geq 2$, then we can write $\frakx=\lc x_1\rc^{(k_1)}\hat{\frakx}$, where $\lc x_1\rc^{(k_1)}\in \tilde{X}$ and $\hat{\frakx}\in \tilde{X}^{i-1}$.
By Lemma~\mref{lem:prodrec}, we get
\begin{equation}
\frakx\diamond \frakx'=\alpha_X(\lc x_1\rc^{(k_1)})
(\hat{\frakx}\diamond
\alpha_X (\frakx')).
\mlabel{eq:frake}
\end{equation}
Then we have
\begin{eqnarray*}
\alpha_X(\frakx\diamond\frakx')&=&
\alpha_X\left(\alpha_X(\lc x_1\rc^{(k_1)})\left(
\hat{\frakx}\diamond\alpha_X(\frakx')\right)\right)
\quad \text{(by Eq.~(\mref{eq:frake}))}\\
&=&
\alpha_X\left(\alpha_X(\lc x_1\rc^{(k_1)})\right)\alpha_X\left(
\hat{\frakx}\diamond\alpha_X(\frakx')\right)
\quad \text{(by Eq.~(\mref{eq:br1}))}\\
&=&\lc x_1\rc^{(k_1)}\alpha_X(\hat{\frakx}\diamond \alpha_X( \frakx'))
\quad \text{(by Eq.~(\mref{eq:bramap}))} \\
&=&\lc x_1\rc^{(k_1)}(\alpha_X(\hat{\frakx}
) \diamond  \frakx')
\quad \text{(by the induction hypothesis and Eq.~(\mref{eq:bramap}))}\\
&=&\left(\alpha_X( \lc x_1\rc^{(k_1)})\alpha_X(\hat{\frakx})\right)
\diamond\alpha_X(\frakx')\quad \text{(by Eqs.~(\mref{eq:bramap}) and ~(\mref{eq:prodrec}))}\\
&=&\alpha_X(\lc x_1\rc^{(k_1)}
\hat{\frakx})
\diamond \alpha_X(\frakx')\quad \text{(by Eq.~(\mref{eq:br1}))}\\
&=&\alpha_X(\frakx)\diamond\alpha_X(\frakx').
\end{eqnarray*}
This completes the inductive proof of Eq.~(\mref{eq:comdia}).

We next verify the Hom-associative law, that is,
\begin{equation}
\alpha_X(\frakx ) \diamond (\frakx'\diamond \frakx'')=(\frakx \diamond \frakx')\diamond \alpha_X( \frakx''),\quad \forall \frakx\in \tilde{X}^i,\frakx'\in \tilde{X}^j,\frakx''\in \tilde{X}^{\ell}, i, j, \ell\geq 1.
\mlabel{eq:homdia}
\end{equation}
We prove Eq.~(\mref{eq:homdia}) by induction on the sum $i+j+\ell\geq 3$. For $i+j+\ell=3$, we have $\frakx,\frakx',\frakx''\in \tilde{X}$.
By Eq.~(\mref{eq:diam}), we have
$\alpha_X(\frakx) \diamond (\frakx'\diamond \frakx'')=\alpha_X(\frakx)\frakx'\frakx''.$
By Eqs.~(\mref{eq:bramap}) and ~(\mref{eq:prodrec}), we have
\begin{equation}
(\frakx\diamond \frakx')\diamond \alpha_X(\frakx'')=(\frakx\frakx')\diamond
\alpha_X(\frakx'')
=\alpha_X(\frakx)(\frakx'\diamond \alpha_X^2(
\frakx''))=\alpha_X(\frakx)\frakx'\frakx''.
\end{equation}
Thus Eq.~(\mref{eq:homdia}) holds.
Assume that Eq.~(\mref{eq:homdia}) has been proved for $i+j+\ell\leq p$ with $p\geq 3$. Consider $\frakx\in\tilde{X}^i$, $\frakx'\in\tilde{X}^j$ and $\frakx''\in \tilde{X}^{\ell}$ with $i+j+\ell=p+1$.
We distinguish two cases, depending on whether or not $i=1$.
\smallskip

\noindent
{\bf Case 1. $i=1$:} Then $\frakx$ is in $\tilde{X}$. So $\alpha_X(\frakx)$ is also in $\tilde{X}$. Then by Eqs.~(\mref{eq:bramap}) and ~(\mref{eq:prodrec}) again, we have
$$
\alpha_X(\frakx) \diamond (\frakx'\diamond \frakx'')=\alpha_X(\frakx)(\frakx'\diamond\frakx'')\quad
\text{and}\quad(\frakx\diamond\frakx')
\diamond\alpha_X(\frakx'')
=(\frakx\frakx')\diamond \alpha_X(\frakx'')=\alpha_X(\frakx)(\frakx'\diamond
\frakx'').
$$
Thus Eq.~(\mref{eq:homdia}) holds.
\smallskip

\noindent
{\bf Case 2. $i\geq 2$:} Then we assume $\frakx=\lc x\rc^{(k)}\hat{\frakx}$, where $\lc x\rc^{(k)}\in \tilde{X}$ and
$\hat{\frakx}\in \tilde{X}^{i-1}$. By the induction hypothesis, we have
\begin{eqnarray*}
\alpha_X( \frakx)\diamond (\frakx'\diamond \frakx'')
&=&\left(\alpha_X(\lc x\rc ^{(k)})\alpha_X(
\hat{\frakx})\right)\diamond(\frakx'\diamond
\frakx'')\quad \text{(by Eq.~(\mref{eq:br1}))}\\
&=&\lc x\rc^{(k)}\left(\alpha_X(\hat{\frakx})\diamond\alpha_X( \frakx'\diamond\frakx'')\right)\quad \text{(by Eqs.~(\mref{eq:bramap}) and ~(\mref{eq:prodrec})) }\\
&=&\lc x\rc^{(k)}\left(\alpha_X(\hat{\frakx})
\diamond(\alpha_X(\frakx')\diamond\alpha_X(
\frakx''))\right)\quad
\text{(by Eq.~(\mref{eq:comdia})) }\\
&=&\lc x\rc^{(k)}\left((\hat{\frakx}
\diamond\alpha_X( \frakx'))\diamond
\frakx''\right)\quad\text{(by the induction hypothesis and  Eq.~(\mref{eq:bramap})) }\\
&=&\left(\alpha_X(\lc x\rc^{(k)})\left(\hat{\frakx}\diamond\alpha_X(
\frakx')\right)\right)\diamond\alpha_X(\frakx'')
\quad
\text{(by Eqs.~(\mref{eq:bramap}) and ~(\mref{eq:prodrec})) }\\
&=&\left(\big(\lc x\rc^{(k)}\hat{\frakx}\big)\diamond \frakx'\right)\diamond\alpha_X( \frakx'')
\quad
\text{(by Eqs.~(\mref{eq:bramap}) and ~(\mref{eq:prodrec})) }\\
&=&(\frakx\diamond\frakx')\diamond\alpha_X( \frakx'').
\end{eqnarray*}
This completes the inductive proof of Eq.~(\mref{eq:homdia}). Thus
$(\calh(X),\diamond,\alpha_X)$ is an involutive Hom-semigroup.

\subsubsection{The proof of Theorem~\mref{thm:freeb}
(\mref{it:freebra})}
We now prove Theorem~\mref{thm:freeb}
(\mref{it:freebra}). Let $(S,\cdot,\beta)$ be an involutive Hom-semigroup. Let $f:X\rightarrow S$ be a set map. We will construct a map $\bar{f}:\calh(X)\rightarrow S$ by defining $\bar{f}(\frakx)$ for all $\frakx\in \calh(X)$. We achieve this by defining $\bar{f}(\frakx)$ for $\frakx\in \tilde{X}^n$ by induction on $n\geq 1$.
For $\frakx\in\tilde{X}$, we have $\frakx=\lc x\rc^{(k)}$, where $x\in X$ and $k\in\{0,1\}$. Then we define
\begin{equation}
\bar{f}(\frakx)=\beta^k(f(x)).
\mlabel{eq:base}
\end{equation}
Suppose $\bar{f}(\frakx)$ have been defined for $\frakx\in\tilde{X}^n$ with $n\geq 1$ and consider $\frakx\in\tilde{X}^{n+1}$. Then we can write $\frakx=\lc x\rc^{(k)}\hat{\frakx}$, where $\lc x\rc^{(k)}\in \tilde{X}$ and $\hat{\frakx}\in\tilde{X}^n$.
Then define
\begin{equation}
\bar{f}(\frakx)=\bar{f}(\lc x\rc^{(k)})\cdot\bar{f}(\hat{\frakx}),
\mlabel{eq:defbar}
\end{equation}
which is well-defined by the induction hypothesis. It remains to prove that the map $\bar{f}$ defined above is indeed a homomorphism of Hom-semigroups. First we prove that $\bar{f}$ satisfies
\begin{equation}
\bar{f}(\alpha_X(\frakx))=\beta(\bar{f}(\frakx))\quad \text{ for all } \frakx\in \tilde{X}^n.
\mlabel{eq:barfbet}
\end{equation}
We prove Eq.~(\mref{eq:barfbet}) by induction on $n\geq 1$. For $n=1$, $\frakx$ is in $\tilde{X}$ and hence is of the form $\frakx=\lc x\rc^{(k)}$. By the definition of $\bar{f}$, we have
$$\bar{f}(\alpha_X(\frakx))=\bar{f}(\lc x \rc^{(\bar{k+1})})=\beta^{\bar{k+1}}
(f(x))=\beta(\beta^k(f(x)))
=\beta(\bar{f}(\frakx)).$$
Assume Eq.~(\mref{eq:barfbet}) has been proved for $n\leq p$ with $p\geq 1$. Let $\frakx=\lc x\rc^{(k)}\hat{\frakx}\in \tilde{X}^{p+1}$, where $\lc x\rc^{(k)}\in\tilde{X}$ and $\hat{\frakx}\in\tilde{X}^p$. By the induction hypothesis, together with Eqs.~(\mref{eq:alphx}),~(\mref{eq:br1}),~(\mref{eq:base}) and~(\mref{eq:defbar}), we have
\begin{eqnarray*}
\bar{f}(\alpha_X(\frakx))&=&\bar{f}(\lc x\rc^{(\bar{k+1})}\alpha_X( \hat{\frakx}))\quad %\text{(by Eqs.~(\mref{eq:alphx}) and~(\mref{eq:br1}))}
\\
&=&\bar{f}(\lc x\rc^{(\bar{k+1})})\cdot\bar{f}(
\alpha_X(\hat{\frakx})) %\quad \text{(by Eq.~(\mref{eq:defbar}) )}
\\
&=&\beta^{\bar{k+1}}(f(x))\cdot
\beta({\bar{f}
(\hat{\frakx})}) %\quad \text{(by the induction hypothesis  and~(\mref{eq:base}))}
\\
&=&\beta\left(\beta^k(f(x))\right)\cdot \beta(\bar{f}(\hat{\frakx}))\\
&=&\beta\left(\beta^k(f(x))\cdot \bar{f}(\hat{\frakx})\right)\\
&=&\beta\left(\bar{f}(\lc x\rc^{(k)})\cdot\bar{f}(\hat{\frakx})
\right)\\
&=&\beta(\bar{f}(\frakx)),%\quad \text{(by Eq.~(\mref{eq:defbar}) )}
\end{eqnarray*}
completing the proof of Eq.~(\mref{eq:barfbet}).

We finally need to verify
\begin{equation}
\bar{f}(\frakx\diamond\frakx')=
\bar{f}(\frakx)\cdot \bar{f}(\frakx')
\quad\text{ for all } \frakx\in\tilde{X}^i, \frakx'\in\tilde{X}^j.
\mlabel{eq:barfdia}
\end{equation}
We use induction on the sum $i+j\geq 2$. When $i+j=2$, we have $\frakx,\frakx'\in \tilde{X}$. By the definition of $\bar{f}$, we have
$$\bar{f}(\frakx\diamond \frakx')=\bar{f}(\frakx\frakx')=\bar{f}
(\frakx)\cdot\bar{f}(\frakx').$$
Assume Eq.~(\mref{eq:barfdia}) holds for $i+j\leq p$ with $p\geq 2$. Let $\frakx\in\tilde{X}^i$ and $\frakx'\in \tilde{X}^j$ with $i+j=p+1$. Then we consider two cases.
\smallskip

\noindent
{\bf Case 1. $i=1$:} Then $\frakx$ is in $\tilde{X}$. By the definition of $\bar{f}$, we have
\begin{equation}
\bar{f}(\frakx\diamond\frakx')=\bar{f}
(\frakx\frakx')=\bar{f}(\frakx)\cdot
\bar{f}(\frakx').
\end{equation}
\smallskip
\noindent
{\bf Case 2. $i\geq 2$:} Then we have $\frakx=\lc x\rc^{(k)}\hat{\frakx}$, where $\lc x\rc^{(k)}\in \tilde{X}$ and $\hat{\frakx}\in \tilde{X}^{i-1}$. Then we have
%by the induction hypothesis and Eqs.~(\mref{eq:prodrec}),~(\mref{eq:defbar}) and~(\mref{eq:barfbet}), we have
\begin{eqnarray*}
\bar{f}(\frakx\diamond \frakx')&=&\bar{f}\left((\lc x\rc^{(k)}\hat{\frakx})\diamond\frakx'
\right)\\
&=&\bar{f}\left(\alpha_X(\lc x\rc^{(k)})(\hat{\frakx}\diamond
\alpha_X( \frakx'))\right)\quad \text{(by Eq.~(\mref{eq:prodrec}) )}\\
&=&\bar{f}\left(\alpha_X(\lc x\rc^{(k)})\right)\cdot
\bar{f}\Big(\hat{\frakx}
\diamond\alpha_X(\frakx')\Big)\quad \text{(by Eq.~(\mref{eq:defbar}) )}\\
&=&\bar{f}\left(\alpha_X(\lc x\rc^{(k)})\right)\cdot
\Big(\bar{f}(\hat{\frakx})\cdot\bar{f}(
\alpha_X(\frakx'))
\Big)\quad \text{(by the induction hypothesis )}\\
&=&\beta\Big(\bar{f}(\lc x\rc^{(k)})\Big)\cdot\Big(\bar{f}(\hat{\frakx})\cdot
\beta(\bar{f}(\frakx'))\Big)\quad \text{(by Eq.~(\mref{eq:barfbet}) )}\\
&=&(\bar{f}(\lc x\rc^{(k)})\cdot \bar{f}(\hat{\frakx}))\cdot\bar{f}(
\frakx')\quad \text{(by Hom-associativity and $\beta^2=\id$ )}\\
&=&\bar{f}(\lc x\rc^{(k)}\hat{\frakx})\cdot\bar{f}(
\frakx')\quad \text{(by Eq.~(\mref{eq:defbar}))}\\
&=&\bar{f}(\frakx)\cdot\bar{f}(\frakx').
\end{eqnarray*}
This completes the induction. The uniqueness of $\bar{f}$ follows from Eqs.~(\mref{eq:base}) and ~(\mref{eq:defbar}). Thus the proof of Theorem~\mref{thm:freeb}(\mref{it:freebra}) is now completed.

\section{Free involutive Hom-associative algebras on a set}
\mlabel{ssec:freeonset}

We will consider the construction of the free involutive Hom-associative algebra on a set.
Let $\bfk$ denote a commutative ring with identity.

\begin{defn}
\begin{enumerate}
\item
A {\bf Hom-associative algebra} is a triple $(A,\cdot,\alpha)$ consisting of a $\bfk$-module $A$, a $\bfk$-linear map $\cdot:A\ot A\rightarrow A$ and a multiplicative linear map  $\alpha:A\rightarrow A$ (namely $\alpha(x\cdot y)=\alpha(x)\cdot \alpha(y)$) satisfying the Hom-associativity
\begin{equation}
\alpha(x)(yz)=(xy)\alpha(z)\quad \forall x,y,z\in A.
\mlabel{eq:homass}
\end{equation}
\item
A Hom-associative algebra $(A,\cdot,\alpha)$ is called {\bf involutive} if $\alpha^2=\id$.
\item
Let $(A,\cdot,\alpha)$ and $(B,\ast,\beta)$ be two Hom-associative algebras. A $\bfk$-linear map $f:A\rightarrow B$ is a {\bf homomorphism of Hom-associative algebras} if
$$f(x\cdot y)=f(x)\ast f(y)\quad\text{and}\quad f(\alpha(x))=\beta(f(x)),\quad\forall\, x,y\in A.$$
\item
A {\bf free involutive Hom-associative algebra on a set $X$} is an involutive Hom-associative algebra $(F(X),\ast,\alpha_X)$ together with a map $j_X:X\rightarrow F(X)$ with the property that,
for any involutive Hom-associative algebra $(A,\cdot,\alpha)$ together with a map $f:X\rightarrow A$, there is a unique homomorphism $\bar{f}:F(X)\rightarrow A$ of Hom-associative algebras such that $f=\bar{f}\circ j_X$.
\delete{
In other words, the following diagram commutes.
\[\xymatrix{
X \ar[rr]^(0.5){j_X} \ar[drr]_{f}
    && F(X) \ar[d]^{\free{f}} \\
&& A } \]
}
\end{enumerate}
\end{defn}
Let $X$ be a given set. Let $(\calh(X),\diamond,\alpha_X)$ be the free involutive Hom-semigroup on the set $X$ obtained in Theorem~\mref{thm:freeb}. Let $\bfk\calh(X)$ be the free $\bfk$-module spanned by $\calh(X)$. We extend the binary operation $\diamond$ and the multiplicative map $\alpha_X$ to $\bfk \calh(X)$ by bilinearity and linearity respectively.
\begin{theorem}
Let $X$ be a given set. Let $j_X:X\rightarrow \bfk\calh(X)$ be the inclusion map.  Then
\begin{enumerate}
\item
The triple $(\bfk\calh(X),\diamond,\alpha_X)$ is an involutive Hom-associative algebra.
\mlabel{it:invx}
\item
The quadruple $(\bfk\calh(X),\diamond,\alpha_X,j_X)$ is the free involutive Hom-associative algebra on $X$.
\mlabel{it:fre}
\end{enumerate}
\mlabel{thm:freex}
\end{theorem}
\begin{proof}
(\mref{it:invx}) By Theorem~\mref{thm:freeb}(\mref{it:hombra}),
Item (\mref{it:invx}) holds.
\smallskip

\noindent
(\mref{it:fre}) We can give a direct proof following the one for Theorem~\mref{thm:freeb}.(\mref{it:freebra}). Alternatively, note that taking the $\bfk$-module span of an involutive Hom-semigroup, that is, taking the free $\bfk$-module of an involutive Hom-semigroup, is the adjoint functor of the forgetful functor from the category $\mathbf{InvHomAs}$ of involutive Hom-associative algebras to the category $\mathbf{InvHomSg}$ of involutive Hom-semigroups, by forgetting the additive structure. On the other hand, the free construction in Theorem~\mref{thm:freeb}.(\mref{it:freebra}) provides the adjoint functor of the forgetful functor from $\mathbf{InvHomSg}$ to the category $\mathbf{Set}$ of sets.

Now the forgetful functor from the category $\mathbf{InvHomAs}$ to the category $\mathbf{Set}$ is the composition of the forgetful functor from $\mathbf{InvHomAs}$ to $\mathbf{InvHomSg}$ and the forgetful functor from $\mathbf{InvHomSg}$ to $\mathbf{Set}$. As is well-known (see for example~\cite[Theorem 1, p. 101]{Ma}), the adjoint functor of a composed functor is the composition of the adjoint functors. Therefore the functor that assigns $X\in \mathbf{Set}$ to $\bfk\calh(X)\in \mathbf{InvHomAs}$
is the adjoint functor of the forgetful functor from $\mathbf{InvHomAs}$ to $\mathbf{Set}$. This proves Item~(\mref{it:fre}).
\end{proof}

In a forthcoming paper~\cite{GZ}, we will continue the study of free Hom-associative algebras and apply it to investigate enveloping Hom-associative algebras~\cite{Yau4} of Hom-Lie algebra, in particular on the Poincar\'e-Birkhoff-Witt type theorem for the enveloping Hom-associative algebras.
\smallskip

\noindent
{\bf Acknowledgements}:
This work is supported by the National Natural Science Foundation of China (Grant No. 11371178) and the National Science Foundation of US (Grant No. DMS~1001855).

%\addcontentsline{toc}{section}{\numberline {}References}
%

\end{document}